\documentclass[12]{article}

\usepackage{latexsym}
\usepackage{amsfonts}
\usepackage{graphicx}
\usepackage{amsmath}
\usepackage{amssymb}
\usepackage{geometry}
\usepackage[latin1]{inputenc}
\usepackage[caption=false]{subfig}
\makeatletter \@addtoreset{equation}{section}

\makeatletter

\def\pa {\partial}

\def\be   {\begin{equation}}   \def\ee   {\end{equation}}
\def\ba   {\begin{array}}      \def\ea   {\end{array}}
\def\bea  {\begin{eqnarray}}   \def\eea  {\end{eqnarray}}
\def\bean {\begin{eqnarray*}}  \def\eean {\end{eqnarray*}}

\newtheorem{theorem} {Theorem}
\newtheorem{lemma}{Lemma}
\newtheorem{definition} {Definition}

\newtheorem{corollary} {Corollary}
\newtheorem{remark}{Remark}

\newcommand{\bo} {\ensuremath{{\bf i_1}}}
\newcommand{\bos}{\ensuremath{{\bf i_1^{\text 2}}}}
\newcommand{\bts}{\ensuremath{{\bf i_2^{\text 2}}}}

\newcommand{\bj}{\ensuremath{{\bf j}}}

\newcommand{\bjs}{\ensuremath{{\bf j^{\text 2}}}}

\newcommand{\eo} {\ensuremath{{\bf e_1}}}
\newcommand{\et} {\ensuremath{{\bf e_2}}}
\newcommand{\bt} {\ensuremath{{\bf i_2}}}

\newcommand{\mC}{\ensuremath{\mathbb{C}}}

\newcommand{\mN}{\ensuremath{\mathbb{N}}}

\newcommand{\mR}{\ensuremath{\mathbb{R}}}
\newcommand{\mT}{\ensuremath{\mathbb{T}}}

\newcommand{\mK}{\ensuremath{\mathcal{K}}}
\newcommand{\mJ}{\ensuremath{\mathcal{J}}}
\newcommand{\BC}{\ensuremath{\mathbb{B}\mathbb{C}}}
\renewcommand{\(}{\left(}
\renewcommand{\)}{\right)}
\newcommand{\oa}{\left\{}
\newcommand{\fa}{\right\}}
\renewcommand{\[}{\left[}
\renewcommand{\]}{\right]}

\begin{document}

\markboth{C. Matteau and D. Rochon}{The Inverse Iteration Method for Julia Sets in the 3-Dimensional Space}

\title{The Inverse Iteration Method for Julia Sets in the 3-Dimensional Space}

\author{Claudia Matteau\thanks{E-mail: \texttt{Claudia.Matteau@UQTR.CA}} \and
Dominic Rochon\thanks{E-mail: \texttt{Dominic.Rochon@UQTR.CA}}}
%\date{\today}
\date{Département de mathématiques et
d'informatique \\ Université du Québec à Trois-Rivières \\
C.P. 500 Trois-Rivières, Québec \\ Canada, G9A 5H7}

\maketitle

\begin{abstract}
In this article, we introduce the adapted inverse iteration method to generate bicomplex Julia sets associated to the polynomial map $w^2 +c $. The result is based on a full characterization of bicomplex Julia sets as the boundary of a particular bicomplex cartesian set and the study of the fixed points of $w^2 + c$. The inverse iteration method is used in particular to generate and display in the usual 3-dimensional space bicomplex dendrites.
\end{abstract}

\vspace{1cm} \noindent\textbf{AMS subject classification:} 37F50, 32A30, 30G35, 00A69\\
\vspace{1cm} \noindent\textbf{Keywords: }Bicomplex dynamics, Inverse iteration method, Generalized Julia sets, 3D fractals\\

%-------------------------------------------------------------------------------

\section{Introduction}

Fractal sets created by iterative processes have been greatly studied in the past decades (see \cite{devaney, douadyHubbard, falconer} and \cite{nishimura}). After being displayed in the complex plane, they became part of the 3-dimensional space when A. Norton \cite{norton} gave straightforward algorithms using iteration with quaternions. The quaternionic Mandelbrot set defined by the quadratic polynomial of the form $q^2 + c$ was explored in \cite{gomatam} and \cite{holbrook}. However, as established in \cite{bedding}, it seems that no interesting dynamics could arise from this approach based on the local rotations of the classical sets. Another set of numbers revealed to be possibly more appropriate : Bicomplex Numbers. In \cite{rochon1}, the author used bicomplex numbers to produce and display in 3D a Mandelbrot set for the quadratic polynomial of the form $w^2 + c$. Filled-in Julia sets were also generated using a method analogous to the classical one in the complex plane \cite{rochon1, rochon2}. Since the bicomplex polynomial $P_{c}(w)=w^2+c$ is the following mapping of
$\mathbb{C}^{2}:\mbox{ }\left( z_1^2-z_2^2+c_1,\mbox{ }2z_1z_2+c_2 \right)$ where $w=z_1+z_2\bold{\bold{i_2}}:=(z_1,z_2)$ and $c=c_1+c_2\bold{\bold{i_2}}:=(c_1,c_2)$, bicomplex dynamics is a particular case of dynamics of several complex variables. More specifically, we note that this mapping is not a holomorphic automorphism of $\mathbb{C}^{2}$.

In this article, we study bicomplex Julia sets associated with the quadratic polynomial $w^2 + c$. We give a specific characterization of bicomplex Julia sets derived from a more general result in terms of the boundary of a bicomplex cartesian set. This characterization allows an easy display in the usual 3D space. The study of the inverse iterates and fixed points of $w^2 + c$ along with the characterization previously introduced lead to the first generalization of the inverse iteration method in two complex variables. This method, well known in the complex plane to generate Julia sets (see \cite{gamelin}, \cite{peitgen1} and \cite{peitgen2}), is used to generate and display in 3D a particular class of bicomplex Julia sets.

%--------------------------------------------------------------------------------------------------------

\section{Preliminaries}

\subsection{Julia Sets in the Complex Plane} \label{Julia2D}

Julia sets in the complex plane are defined according to the behavior of the forward iterates of a rational function. In this article, we restrict our study of Julia sets to a polynomial map that is easy to work with and has a dynamical system equivalent to the one of any polynomial map of degree two : $P_c(z) = z^2 + c$ where $z,c \in \mC$ and $c$ is fixed. First, we consider its iterates and fixed points. Next, we present some important and well known results about Julia sets.

The forward iterates of $P_c$ are given by $P_c ^0 (z) = z$ and $P_c ^n (z) = (P_c (z))^{\circ n} = (P_c \circ P_c ^{(n-1)})(z)$ for $ n \in \{1,2, \dots\}$. The inverse iterates are defined as $P_c ^{-1}(z) := (P_c (z))^{\circ (-1)} = \{w \in \mC \mid P_c(w) = z\}$ and $P_c ^{-m}(w) := (P_c ^m (w))^{\circ (-1)}$ for $m \in \{ 1,2, \dots\}$. The multivalued function $\sqrt{z-c}$ is associated to $P_c ^{-1}$.
The fixed points of $P_c$ are found by solving the equation $P_c (z_0) = z_0$. A fixed point $z_0$ is said to be attractive if $0 \leq |2 z_0| < 1$, repelling if $|2 z_0| > 1$ and indifferent if $|2 z_0| = 1$. For $c = \frac{1}{4}$, there is a single indifferent fixed point $z_0 = \frac{1}{2}$. Otherwise, there are two distinct fixed points and at least one of them is repelling (see \cite{peitgen}).

Let $\mK_c = \{z \in \mC \mid \{P_c ^n (z)\}_{n=0}^{\infty} \textrm{ is bounded }\}$ be the filled-in Julia set associated to $P_c$. The Julia set related to $P_c$ is denoted by $\mJ_c$ and defined as either one of the following :
\begin{enumerate}
  \item The boundary of the filled-in Julia set : $\mJ_c = \pa \mK_c$;
  \item The set of points $z \in \mC$ for which the forward iterates do not form a normal family at $z$ (see \cite{schiff} for details on normal families of functions).
\end{enumerate}

The second definition leads to the following theorem that justifies the inverse iteration method. Note that it is stated for Julia sets $\mJ_P$ defined by any monic polynomial map $P$ of degree $d \geq 2$. It is so valid for $P_c$. In \cite{moi}, the classical statement of the result has been slightly modified from the one in \cite{gamelin}.
\begin{theorem}
  Let $P$ be a monic complex polynomial of degree $d \geq 2$.
  \begin{description}
    \item (i) If $z_0 \in \mJ_P$ and $V$ is any open neighborhood of $z_0$, then for any whole number $k_1 \geq 0$ there exists $N > k_1$ such that
    $\mJ_P \subseteq \bigcup\limits_{k=k_1}^{N} P^k(V) $.
    \item (ii) For any $z_1 \in \mJ_P$, the set of inverse iterates $\oa \bigcup\limits_{k=k_1}^{\infty} P^{-k}(z_1)\fa$ is dense in $\mJ_P$ for all whole number $k_1 \geq 1$.
  \end{description}
  \label{TheoInvC}
\end{theorem}
To generate and display $\mJ_c$ in the complex plane, it suffices to take $z_1 \in \mJ_c$ and compute its inverse iteratively, up to a maximum number of iterations. Since the inverse is given by the complex square root function $\sqrt{z-c}$ which is multivalued, two different approaches may be used. The first one is to compute all branches of the inverse at every iteration, leading to a great number of points generated. The second option is to randomly choose one of the branches of the inverse at each iteration and compute only this one. This last approach seems more appropriate for it is faster and requires less memory space.

From \cite{devKeen}, it is known that $\mJ_c$ is the closure of the set of repelling periodic points of $P_c$. Hence, for a starter $z_1$, one may choose a repelling fixed point of the polynomial map if $c \neq \frac{1}{4}$. If $c = \frac{1}{4}$, then $z_1 = \frac{1}{2}$ is a good starter for the algorithm since it is the only fixed point of $P_c$ and known to be in $\mJ_c$ from \cite{peitgen}. For a good approximation of $\mJ_c$, a high enough number of iterations is needed.

\begin{figure}[!h]
  \centering
  \subfloat[$c = 0,25$]{\includegraphics[width=5cm]{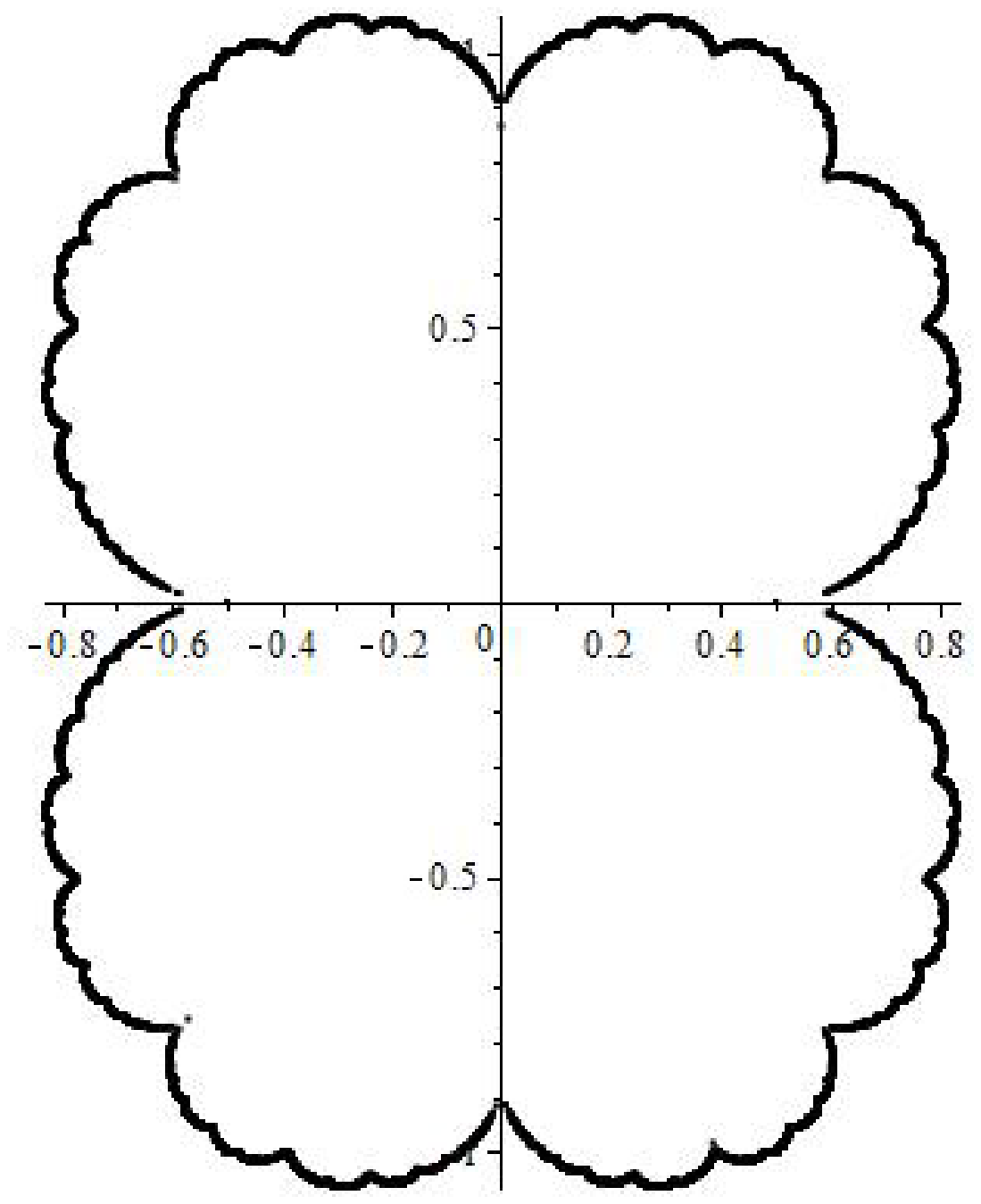}}
  \subfloat[$c = -0,123 + 0,745 \bf{i}$]{\includegraphics[width=5cm]{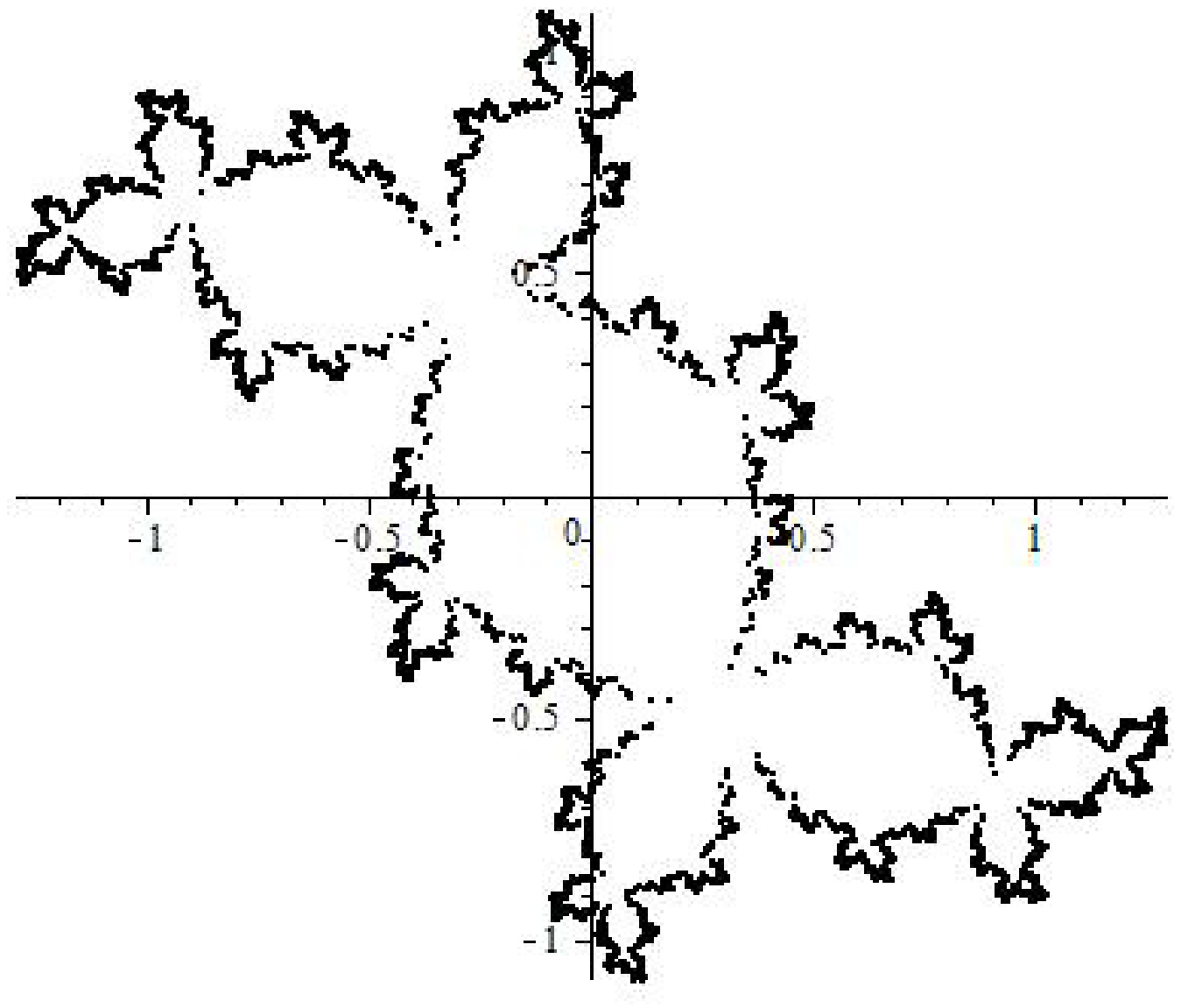}}
  \subfloat[$c = \bf{i}$]{\includegraphics[width=5cm]{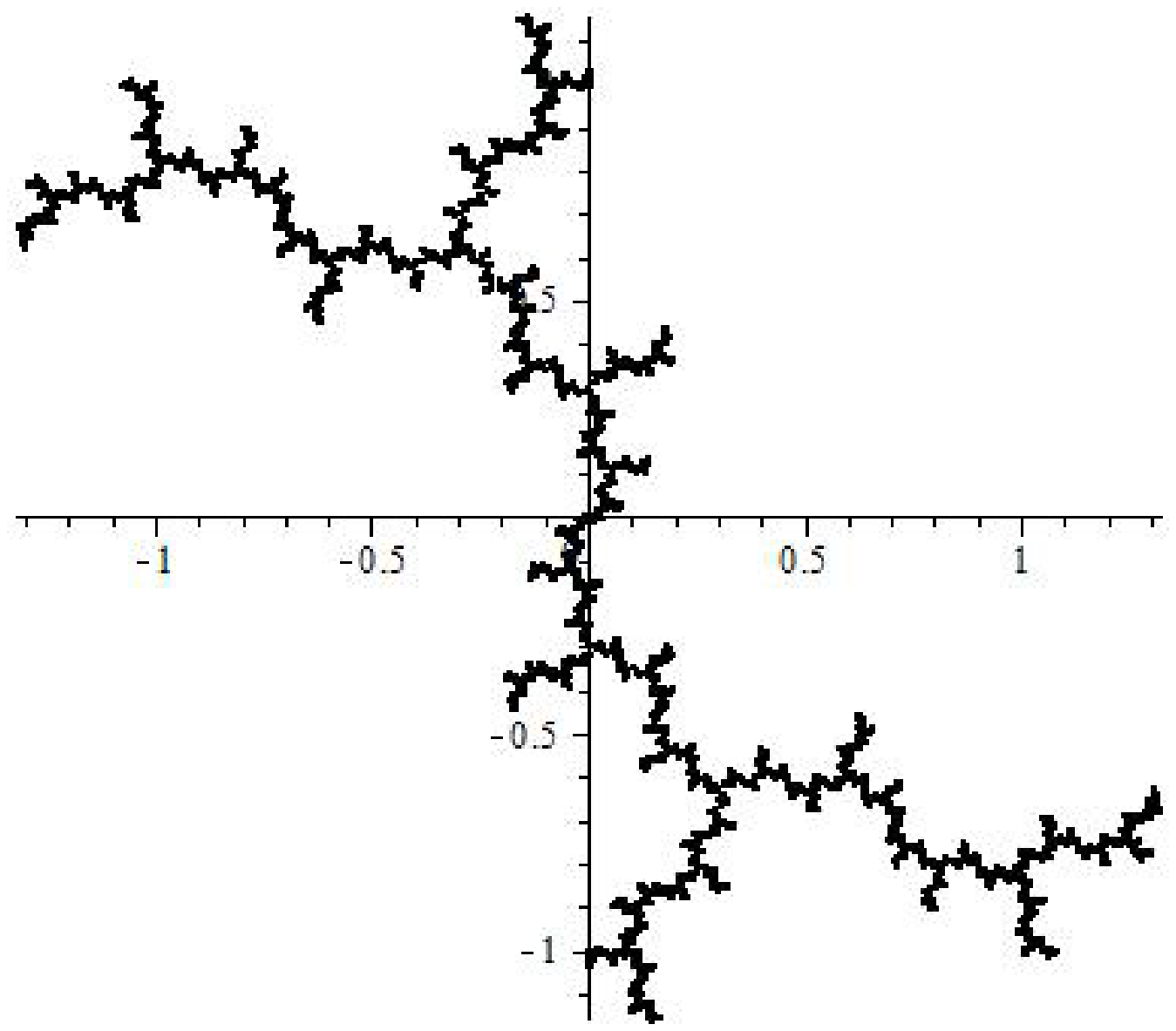}}
  \caption{Julia sets of the complex plane}
  \label{ImJulia2d}
\end{figure}

Images from Figure \ref{ImJulia2d} are those of classical Julia sets produced by the inverse iteration method. For $c = \bf{i}$, $\mK_c$ is known to be a dendrite that is compact set, pathwise connected, locally connected, nowhere dense and that does not separate the plane \cite{carlesonGamelin}. A dendrite set is equal to its boundary and so $\mK_c = \mJ_c$.

\subsection{Bicomplex numbers}

As presented in \cite{price, rochon1} and \cite{rochon2}, bicomplex numbers arise from the work of Corrado Segre \cite{segre} and are defined as follows :
$$\BC = \{ a + b \bo + c \bt + d \bj \mid a,b,c,d \in \mR \}$$
where $\bos = \bts = -1, \bjs = 1, \bo \bt = \bt \bo = \bj, \bt \bj = \bj \bt = - \bo$ and $\bo \bj = \bj \bo = - \bt$. Since we can write $a + b \bo + c \bt + d \bj$ as $ (a + b \bo) + (c + d \bo) \bt$, the set of bicomplex numbers can be seen as
$$ \BC = \{ z_1 + z_2 \bt \mid z_1, z_2 \in \mC (\bo)\}$$
where $\mC (\bo)$ is the set of complex numbers with imaginary unit $\bo$ : $\mC (\bo) = \{ x + y \bo \mid x, y \in \mR \textrm{ and } \bos = -1 \}$. Hence, $\BC$ corresponds to a kind of complexification of the usual complex numbers. From \cite{price}, we can easily see that it is a commutative unitary ring. The set of bicomplex numbers is also sometimes denoted in the literature by $\mC _2 $, $\mT$, $\mathbb{M}(2)$ or by the following complex Clifford algebras ${\rm Cl}_{\Bbb{C}}(1,0) \cong {\rm Cl}_{\Bbb{C}}(0,1)$.\\

An important property of bicomplex numbers is the unique representation using the idempotent elements $\eo = \frac{1+\bj}{2}$ and $\et = \frac{1-\bj}{2}$. In fact, $\forall w=z_1 + z_2 \bt \in \BC$, we have
\begin{eqnarray*}
z_1 + z_2 \bt &=&(z_1 - z_2 \bo) \eo + (z_1 + z_2 \bo) \et\\
              &=&\mathcal{P}_{1}(w)\eo+\mathcal{P}_{2}(w)\et
\end{eqnarray*}
where the projections
$\mathcal{P}_1,\mathcal{P}_2:\BC\longrightarrow\mathbb{C}(\bold{i_1})$ are defined as
$\mathcal{P}_1(z_1+z_2\bold{i_2}):=z_1-z_2\bold{i_1}$ and
$\mathcal{P}_2(z_1+z_2\bold{i_2}):=z_1+z_2\bold{i_1}$.

The usual operations of addition and multiplication can be done term-by-term using this representation :
\begin{description}
  \item \emph{(i)} $(z_1 + z_2 \bt) + (s_1 + s_2 \bt) = \[ (z_1 - z_2 \bo) + (s_1 - s_2 \bo)\] \eo + \[(z_1 + z_2 \bo) + (s_1 + s_2 \bo)\] \et$
  \item \emph{(ii)} $(z_1 + z_2 \bt) \cdot (s_1 + s_2 \bt) = \[(z_1 - z_2 \bo)(s_1 - s_2 \bo)\]\eo + \[(z_1 + z_2 \bo)(s_1 + s_2 \bo)\]\et$
  \item \emph{(iii)} $(z_1 + z_2 \bt) ^n = (z_1 - z_2 \bo)^n \eo + (z_1 + z_2 \bo)^n \et$ for $n = 0, 1, 2, \ldots$.
\end{description}

The real modulus of $w = z_1 + z_2 \bt \in \BC$ is given by $|| w || = \sqrt{|z_1|^2 + |z_2|^2}$ where $| \cdot |$ is the Euclidian norm in $\mC (\bo)$. Writing $z_1 = a + b \bo$ and $z_2 = c + d \bo$, we have $||w|| = \sqrt{a^2 + b^2 + c^2 + d^2}$ which is the Euclidian norm in $\mR ^4$.\\

The square root of a bicomplex number $w=z_1 + z_2 \bt$ is given in terms of the complex square roots of its idempotent components :
\begin{eqnarray*}
\sqrt{z_1 + z_2 \bt}&=& \sqrt{z_1 - z_2 \bo} \; \eo + \sqrt{z_1 + z_2 \bo} \; \et\\
                    &=& \sqrt{\mathcal{P}_{1}(w)} \; \eo + \sqrt{\mathcal{P}_{2}(w)} \; \et.
\end{eqnarray*}
Consequently, it is a multivalued function that admits up to four possible values for every $z_1 + z_2 \bt \in \BC$. The following property using the idempotent representation is proven in \cite{price} and very convenient to demonstrate our main results.
\begin{theorem}
  Let $w = z_1 + z_2 \bt \in \BC$. Then $||w|| = || z_1 + z_2 \bt|| = \sqrt{\frac{|z_1 - z_2 \bo|^2 + |z_1 + z_2 \bo|^2}{2}}$.
  \label{TheoNormBC}
\end{theorem}
Finally, we introduce some classic bicomplex sets along with a result from \cite{price} that brings them together.

\begin{definition}
  Let $X_1, X_2 \subseteq \mC(\bo)$ be nonempty. A bicomplex cartesian set $X \subseteq \BC$ determined by $X_1$ and $X_2$ is
  $$X = X_1 \times_e X_2 := \{ w \in \BC \mid w = w_1 \eo + w_2 \et \textrm{ where } (w_1, w_2) \in X_1 \times X_2 \}.$$
\end{definition}

\begin{definition}
Let $ a = a_1\eo + a_2 \et \in \BC$ and $r_1, r_2 > 0$ be fixed. We define
\begin{description}
  \item (i) The open ball with center $a$ and radius $r_1$ as
  $$B^2 (a,r_1) = \{w \in \BC \mid ||w-a|| < r_1 \}.$$
  \item (ii) The open discus with center $a$ and radii $r_1$ and $r_2$ as
  $$ D(a;r_1, r_2) = B^1 (a_1, r_1) \times_e B^1 (a_2, r_2)$$
  where $B^1(a_1, r_1)$ is the open ball with center $a_1$ and radius $r_1$ in $\mC(\bo)$.
\end{description}
\end{definition}

\begin{theorem}
  Let $a \in \BC$ and $ 0 < r_1 \leq r_2$. Then $B^2 \( a, \frac{r_1}{\sqrt{2}} \) \varsubsetneq D (a; r_1, r_2) \varsubsetneq B^2 \( a, \sqrt{\frac{r_1 ^2 + r_2 ^2}{2}} \).$
  \label{TheoInclBC}
\end{theorem}

\begin{remark}
The set $\mathcal{NC}$ of zero divisors of $\BC$, called
the {\em null-cone}, is given by $\{z_1+z_2\bt\ |\
z_{1}^{2}+z_{2}^{2}=0\}$, which can be rewritten as \be
\mathcal{NC}=\{z(\bo\pm\bt)|\ z\in \mathbb{C}(\bo)\}. \ee
\end{remark}

\subsection{Bicomplex Dynamics}

We now introduce basic results in bicomplex dynamics as presented in \cite{rochon1}, \cite{rochon2} and \cite{vincent}. The definition of bicomplex Julia sets is also given.\\

Set $P_c (w) = w^2 + c $ where $w, c \in \BC$ and $c$ is fixed. The forward iterates of $P_c$ are defined as in the complex case : $P_c ^0 (w) = w$ and $P_c ^n (w) = (P_c (w))^{\circ n} = (P_c \circ P_c ^{(n-1)})(w)$ for $ n \in \{1,2, \dots\}$. Also, for $w = z_1 + z_2 \bt$ and $c = c_1 + c_2 \bt$, one can show using induction that
$$ P_c ^n (w) = P_{c_1 - c_2 \bo} ^n (z_1 - z_2 \bo) \eo + P_{c_1 + c_2 \bo} ^n (z_1 + z_2 \bo) \et \quad \textrm{ for all } n \in \mN.$$
The inverse iterates are defined as $P_c ^{-1}(w) := (P_c (w))^{\circ (-1)} = \{ w_0 \in \BC \mid P_c (w_0) = w \}$ and $P_c ^{-m}(w) := (P_c ^m (w))^{\circ (-1)}$ for $m \in \{ 1,2, \dots\}$. Again, the square root function $\sqrt{w-c}$ is associated to $P_c ^{-1}$. Since it is multivalued, the inverse iterates correspond to a set of bicomplex numbers rather than a single value. However, it may still be written as
\begin{eqnarray*}
  P_c ^{-n} (w) & = & P_{c_1 - c_2 \bo} ^{-n} (z_1 - z_2 \bo) \eo + P_{c_1 + c_2 \bo} ^{-n} (z_1 + z_2 \bo) \et \\
   & = & P_{c_1 - c_2 \bo} ^{-n} (z_1 - z_2 \bo) \times_e P_{c_1 + c_2 \bo} ^{-n} (z_1 + z_2 \bo) \quad \textrm{ for all } n \in \{1,2, \dots\}.
\end{eqnarray*}
See \cite{moi} for complete details and proof. As in section \ref{Julia2D}, bicomplex filled-in Julia sets are defined according to the behavior of the forward iterates of $P_c$.

\begin{definition}
  The bicomplex filled-in Julia set corresponding to $c \in \BC$ is defined as
  $$ \mK_{2,c} = \{ w \in \BC \mid \{ P_c ^n (w)\}_{n=0}^{\infty} \; \textrm{is bounded} \}.$$
\end{definition}
This set corresponds to a particular bicomplex cartesian set determined by the idempotent components of $c$, as proven in \cite{rochon1}.

\begin{theorem}
  Let $c = c_1 + c_2 \bt \in \BC$. Then $\mK_{2,c} = \mK_{c_1 - c_2 \bo} \times_e \mK_{c_1 + c_2 \bo}$.
\end{theorem}

\begin{definition}
  The bicomplex Julia set corresponding to $c \in \BC$ is the boundary of the associated filled-in Julia set : $\mJ_{2,c} = \pa \mK_{2,c}$.
\end{definition}
From the previous theorem, it is the boundary of a specific bicomplex cartesian set. This is investigated more deeply in this next section.

%-------------------------------------------------------------------------------------------------------

\section{Characterization of bicomplex Julia sets}
\subsection{Boundary of bicomplex cartesian sets}
We first demonstrate two lemmas derived from \cite{price}. Then, we establish a full characterization of the boundary of any bicomplex cartesian set which is valid for Julia sets.

\begin{lemma}
  Let $X_1, X_2 \subseteq \mC (\bo)$ be nonempty and $X \subseteq \BC$ such that $X = X_1 \times_e X_2$. If $z_1 + z_2 \bt \in \pa X$, then $z_1 - z_2 \bo \in \pa X_1$ or $z_1 + z_2 \bo \in \pa X_2$.
  \label{Lemma1}
\end{lemma}

\begin{proof}
Let $w = z_1 + z_2 \bt \in \pa X$. From the idempotent representation, we may write $w = w_1 \eo + w_2 \et$ with $w_1 = z_1 - z_2 \bo$ and $w_2 = z_1 + z_2 \bo$. By the definition of the boundary of $X$, we know that
$$ \forall r >0, \quad \underbrace{B^2(w,r)\cap X \neq \emptyset}_{(1)} \quad \textrm{ and } \quad \underbrace{B^2(w,r)\cap X^c \neq \emptyset}_{(2)}.$$
We need to show that
\begin{eqnarray*}
\forall r > 0, \; B^1(w_1,r)\cap X_1 \neq \emptyset & \textrm{ and } & B^1(w_1,r)\cap X_1^c \neq \emptyset\\
\textrm{or} \quad \forall r > 0, \; B^1(w_2,r)\cap X_2 \neq \emptyset & \textrm{ and } & B^1(w_2,r)\cap X_2^c \neq \emptyset.
\end{eqnarray*}
Set $r>0$. From $(1)$, there exists $s \in X$ such that $s \in B^2 (w, \frac{r}{\sqrt{2}})$. It follows that $s = s_1 \eo + s_2 \et$ with $s_1 \in X_1$ and $s_2 \in X_2$ and $||s-w|| = || (s_1 - w_1) \eo + (s_2 - w_2) \et|| < \frac{r}{\sqrt{2}}$ implies $\sqrt{\frac{|s_1 - w_1|^2 + |s_2 - w_2|^2}{2}} < \frac{r}{\sqrt{2}}$ from Theorem \ref{TheoNormBC}. Thus, $| s_1 - w_1 | < r$ and $|s_2 - w_2| < r$. Therefore, $s_1 \in B^1 (w_1, r) \cap X_1$ and $s_2 \in B^1 (w_2, r) \cap X_2$. \\

From $(2)$, there exists $t \in X^c$ such that $t \in B^2(w, \frac{r}{\sqrt{2}})$. Along the same reasoning as before, we may write $t = t_1 \eo + t_2 \et$ with $t_1 \in B^1 (w_1, r)$ and $t_2 \in B^1 (w_2, r)$. Since $t$ does not belong to the bicomplex cartesian set $X$, there are three possible situations regarding $t_1$ and $t_2$ :
\begin{description}
  \item \emph{(a)} If $t_1 \not\in X_1$ and $t_2 \in X_2$, then $t_1 \in B^1 (w_1, r) \cap X_1 ^c$.
  \item \emph{(b)} If $t_1 \in X_1$ et $t_2 \not\in X_2$, then $t_2 \in B^1(w_2, r) \cap X_2^c$.
  \item \emph{(c)} If $t_1 \not\in X_1$ and $t_2 \not\in X_2$, then $t_1 \in B^1(w_1, r) \cap X_1^c$ and $t_2 \in B^1(w_2, r) \cap X_2^c$.
\end{description}
Thus for all $r>0$,
\begin{eqnarray}
B^1(w_1,r)\cap X_1 \neq \emptyset & \textrm{ and } & B^1(w_1,r)\cap X_1^c \neq \emptyset \label{cas1}\\
\textrm{or} \quad B^1(w_2,r)\cap X_2 \neq \emptyset & \textrm{ and } & B^1(w_2,r)\cap X_2^c \neq \emptyset. \label{cas2}
\end{eqnarray}
Consider the following infinite set : $\{ r_n = \frac{1}{n} \mid n = 1,2, \dots \}$. For these $r$ values, one of the two previous cases has to be verified an infinite number of times. Suppose, without loss of generality, it is case (\ref{cas1}). Then, there exists an infinite subsequence $\{ r_{n_k} \mid k = 1, 2 , \dots \}$ such that $B^1 (w_1, r_{n_k}) \cap X_1 \neq \emptyset$ and $B^1 (w_1, r_{n_k}) \cap X_1 ^c \neq \emptyset$ for all $k \in \{1,2, \dots\}$. Consequently, for any $r' > 0$, there exists $n', k \in {1,2, \dots}$ such that $\frac{1}{n'} < r'$ and $ r_{n_k} = \frac{1}{n_k} < \frac{1}{n'}$ meaning $B^1(w_1, r')\cap X_1 \neq \emptyset$ and $B^1(w_1, r') \cap X_1 ^c \neq \emptyset$. Therefore, the desired property is valid for all $r>0$ and $w_1 \in \pa X_1$. Similarly, we have $w_2 \in \pa X_2$ when case (\ref{cas2}) is investigated. Hence, we conclude $w_1 \in \pa X_1$ or $w_2 \in \pa X_2$.
\end{proof}$\Box$\\

\begin{lemma}
  Let $X_1, X_2 \subseteq \mC (\bo)$ be nonempty and $X \subseteq \BC$ such that $X = X_1 \times_e X_2$.
  \begin{description}
    \item (i) If $w_1 ^0 \in \pa X_1$, then $w \in \pa X$ for all $w \in \{w_1 ^0\} \times_e X_2$.
    \item (ii) If $w_2 ^0 \in \pa X_2$, then $w \in \pa X$ for all $w \in X_1 \times_e \{w_2 ^0\}$.
  \end{description}
  \label{Lemma2}
\end{lemma}

\begin{proof}
We prove the first statement of the lemma. Let $w_1 ^0 \in \pa X_1$ and $w = w_1 ^0 \eo + w_2 \et$ for some $w_2 \in X_2$. Since $w_1 ^0 \in \pa X_1$, then
$$\forall r > 0, \quad \underbrace{B^1(w_1^0,r)\cap X_1 \neq \emptyset}_{(3)} \quad \textrm{ and } \quad \underbrace{B^1(w_1^0,r)\cap X_1^c \neq \emptyset}_{(4)}.$$
Set $r>0$. From $(3)$, there exists $s_1 \in X_1$ such that $s_1 \in B^1 (w_1 ^0, r)$. Set $s = s_1 \eo + w_2 \et \in X$. Then
\begin{eqnarray*}
  || s - w || & = & \sqrt{\frac{|s_1-w_1^0|^2 + |w_2-w_2|^2}{2}}\\
    & \leq & \sqrt{|s_1-w_1^0|^2}\\
    & < & r.
\end{eqnarray*}
Therefore, $s \in B^2 (w,r) \cap X$ and $B^2 (w,r) \cap X \neq \emptyset$. From $(4)$, there exists $t_1 \not\in X_1$ such that $t_1 \in B^1 (w_1 ^0, r)$. So $t = t_1 \eo + w_2 \et \not\in X$. Again we get $|| t - w || \leq | t_1 - w_1 ^0 | < r$ and $t \in B^2 (w,r) \cap X^c$. Hence, for all $ r >0$ we have $B^2(w,r) \cap X \neq \emptyset$ and $B^2 (w,r) \cap X^c \neq \emptyset$, that is $w \in \pa X$. The proof of the second statement goes along the same lines as this one.
\end{proof}$\Box$\\

We can now state the following theorem that directly leads to a useful point of view on bicomplex Julia sets.

\begin{theorem}
  Let $X_1, X_2 \subseteq \mC (\bo)$ be nonempty and $X \subseteq \BC$ such that $X = X_1 \times_e X_2$. The boundary of $X$ if the union of three bicomplex cartesian sets:
  $$\pa X = \(\pa X_1 \times_e X_2\) \cup \(X_1 \times_e \pa X_2\) \cup \(\pa X_1 \times_e \pa X_2\).$$
\end{theorem}

\begin{proof}
First, we demonstrate that $\pa X \subseteq \(\pa X_1 \times_e X_2\) \cup \(X_1 \times_e \pa X_2\) \cup \(\pa X_1 \times_e \pa X_2\)$. To begin with, note that Lemma \ref{Lemma1} may be rewritten as
$$\pa X \subseteq \(\pa X_1 \times_e \mC (\bo)\) \cup \( \mC(\bo) \times_e \pa X_2\) \cup \( \pa X_1 \times_e \pa X_2\).$$
Let $w = w_1 \eo + w_2 \et \in \pa X$. According to the last expression, one of these three situations presents itself :
\begin{description}
  \item \emph{(a)} $w_1 \in \pa X_1$ and $w_2 \in \mC (\bo)$
  \item \emph{(b)} $w_1 \in \mC (\bo)$ and $w_2 \in \pa X_2$
  \item \emph{(c)} $w_1 \in \pa X_1$ and $w_2 \in \pa X_2$.
\end{description}
We need to replace $\mC (\bo)$ by $X_2$ and $X_1$ respectively in the first two cases.
\begin{description}
  \item \emph{(a)} Suppose $w_1 \in \pa X_1$ and $w_2 \in \mC (\bo)$. We must have $w_2 \in \pa X_2$ or $w_2 \not\in \pa X_2$. If $w_2 \in \pa X_2$, then $w \in \pa X_1 \times_e \pa X_2$ which corresponds to \emph{(c)}. Suppose $w_2 \not\in \pa X_2$ and $w_2 \not\in X_2$. For all $r>0$, $B^1(w_2, r) \cap X_2 ^c \neq \emptyset$. Since $w \in \pa X$, from the first part of the proof of Lemma \ref{Lemma1}, we know that $B^1(w_2, r) \cap X_2 \neq \emptyset$ for all $r>0$. Hence, $w_2 \in \pa X_2$ which is a contradiction with our assumption. Therefore, $w_2 \in X_2$ and $w \in \pa X_1 \times_e X_2$.
  \item \emph{(b)} Suppose $w_1 \in \mC (\bo)$ and $w_2 \in \pa X_2$. We must have $w_1 \in \pa X_1$ or $w_1 \not\in \pa X_1$. If $w_1 \in \pa X_1$, then $w \in \pa X_1 \times_e \pa X_2$ which corresponds to \emph{(c)}. If $w_1 \not\in \pa X_1$, then as in \emph{(a)} we conclude that $w_1 \in X_1$ and $w \in X_1 \times_e \pa X_2$.
\end{description}
Consequently, the possible outcomes become $w \in \pa X_1 \times_e X_2$, $w \in X_1 \times_e \pa X_2$ or $w \in \pa X_1 \times_e \pa X_2$ and the first part of the proof is complete. \\

On the other hand, we determine that $ \(\pa X_1 \times_e X_2 \) \cup \(X_1 \times_e \pa X_2\) \cup \( \pa X_1 \times_e \pa X_2\) \subseteq \pa X$. From Lemma \ref{Lemma2}, we know that $\( \pa X_1 \times_e X_2 \) \cup \( X_1 \times_e \pa X_2 \) \subseteq \pa X$. We only need to show that if $ w \in \pa X_1 \times_e \pa X_2$, then $w \in \pa X$. Let $w = w_1 \eo + w_2 \et$ with $w_1 \in \pa X_1$ and $w_2 \in \pa X_2$. We use the definition of the boundary of a set with converging sequences. There exists sequences $\{ p_n ^1 \}_{n=1}^{\infty} \subseteq X_1$ and $\{ q_n ^1 \}_{n=1}^{\infty} \subseteq X_1 ^c$ that both converge to $w_1$. Likewise, there exists sequences $\{ p_n ^2 \}_{n=1}^{\infty} \subseteq X_2$ and $\{ q_n ^2 \}_{n=1}^{\infty} \subseteq X_2 ^c$ that both converge to $w_2$. Therefore, sequences $\{ p_n := p_n ^1 \eo + p_n ^2 \et \}_{n=1}^{\infty} \subseteq X$ and $\{ q_n := q_n ^1 \eo + q_n ^2 \et\}_{n=1}^{\infty} \subseteq X^c$ both converge to $w$. Hence, $w \in \pa X$ and $\( \pa X_1 \times_e X_2\) \cup \( X_1 \times_e \pa X_2 \) \cup \( \pa X_1 \times_e \pa X_2 \) \subseteq \pa X$.
\end{proof}$\Box$\\

This theorem allows a better understanding of the structure of bicomplex Julia sets, as given in the next corollary. We see that the set is solely determined by complex Julia sets and filled-in Julia sets associated to the idempotent components of $c$.

\begin{corollary}
  The bicomplex Julia set corresponding to $c = c_1 + c_2 \bt \in \BC$ may be expressed as
  $$\mJ_{2,c} = \( \mJ_{c_1 - c_2 \bo} \times_e \mK_{c_1 + c_2 \bo} \) \cup \( \mK_{c_1 - c_2 \bo} \times_e \mJ_{c_1 + c_2 \bo} \)$$
  \label{corJulia}
  where $\mJ_{c_1 - c_2 \bo} \times_e \mJ_{c_1 + c_2 \bo}\subseteq \mJ_{2,c}$.
\end{corollary}

This corollary provides a useful insight on bicomplex Julia sets that leads to an easy display in the usual 3D space. Indeed, for a certain $c = c_1 + c_2 \bt \in \BC$, it suffices to generate points of $\mJ_{c_1 - c_2 \bo}$, $\mJ_{c_1 + c_2 \bo}$, $\mK_{c_1 - c_2 \bo}$ and $\mK_{c_1 + c_2 \bo}$ and create the bicomplex cartesian sets of the characterization using the idempotent representation. Rewriting the elements of the set under the representation $a + b \bo + c \bt + d \bj$, a 3D cut is applied by keeping only the points for which the absolute value of $d$ (the $\bj$-component) is less than a certain value $\epsilon > 0$. The display of the other components in the 3D space produces an approximation of $\mJ_{2,c}$ that we may observe.\\

Three examples are given in Figure \ref{ImJulia3d1}. Note that for images (a) and (b), shapes of complex Julia sets from Figure \ref{ImJulia2d} are easily noticeable. For $c = 0,0635 + 0,3725 \bo + 0,3725 \bt + 0,1865 \bj$, we have $c_1 - c_2 \bo = 0,25$ and $c_1 + c_2 \bo = -0,123+0,745 \bo$. Both complex shapes may be seen when moving the object around. The color black represents the set $\mJ_{c_1 - c_2 \bo} \times_e \mJ_{c_1 + c_2 \bo}$ from Corollary \ref{corJulia} (the one determined by complex Julia sets) and provides the main structure of the set.

\begin{figure}[!h]
  \centering
  \subfloat[$c = 0,25$]{\includegraphics[width=7cm]{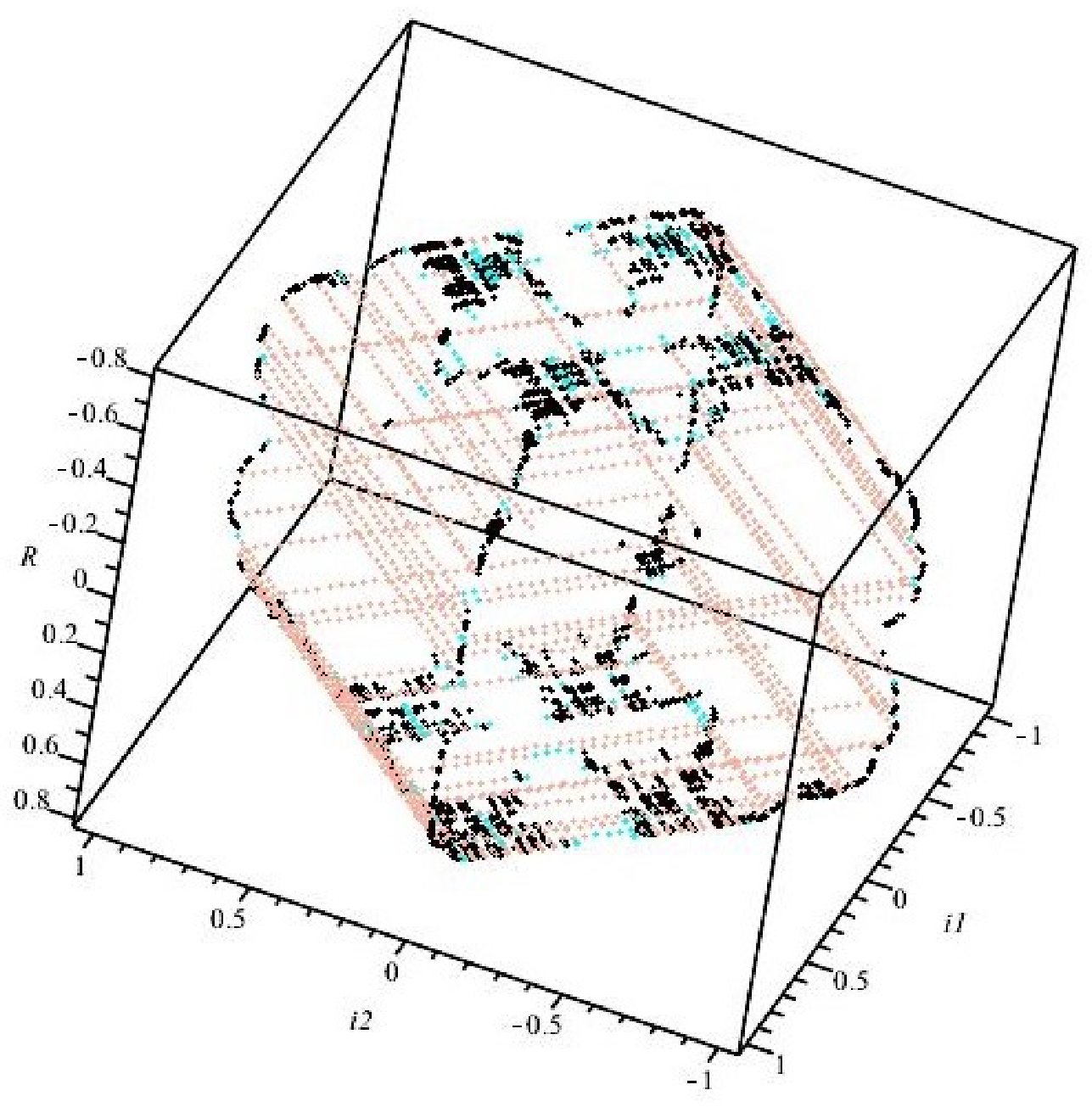}}
  \subfloat[$c = -0,123+0,745\bo$]{\includegraphics[width=7cm]{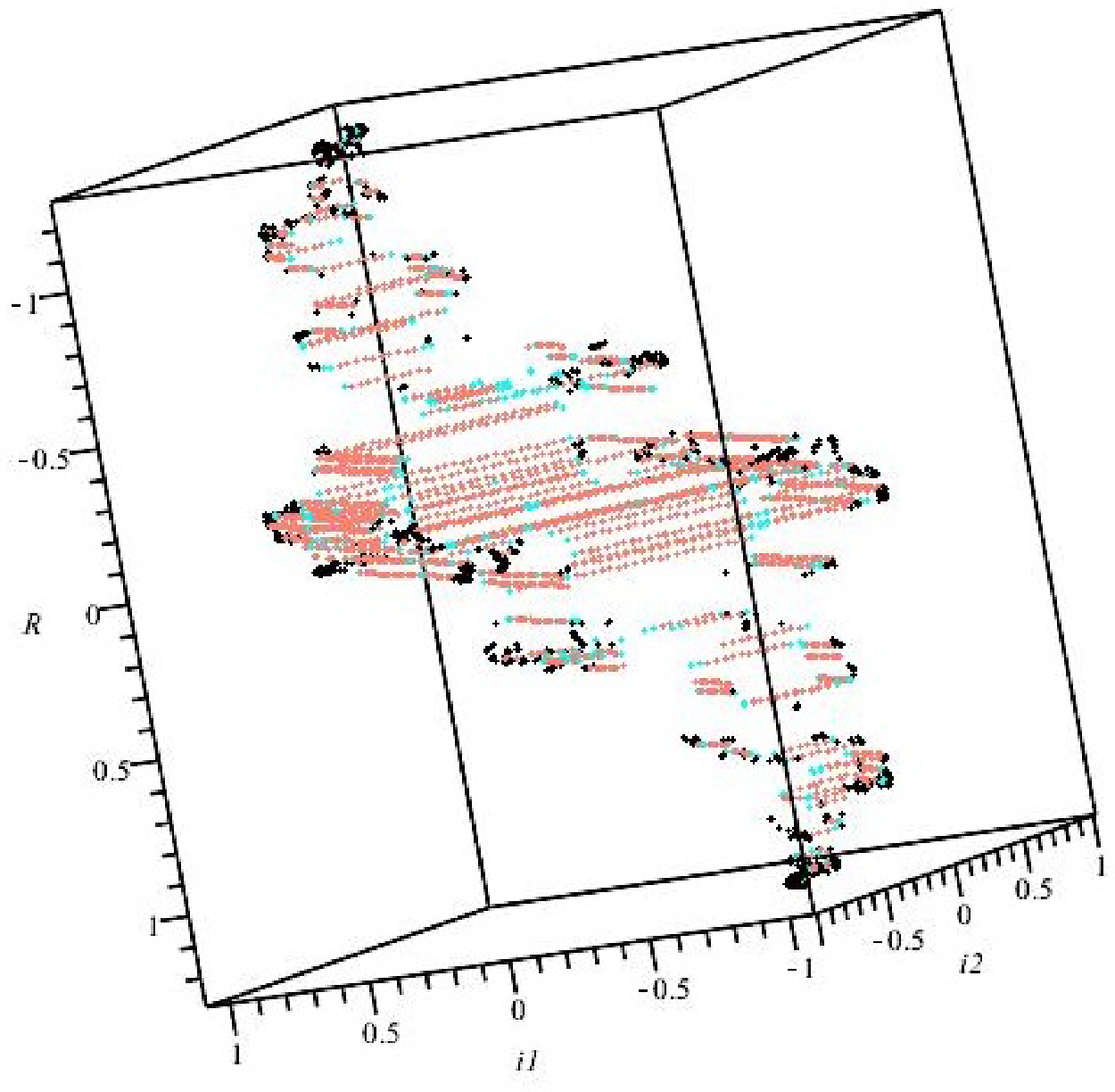}}\\
  \subfloat[$c = 0,0635 + 0,3725 \bo + 0,3725 \bt + 0,1865 \bj$]{\includegraphics[width=8cm]{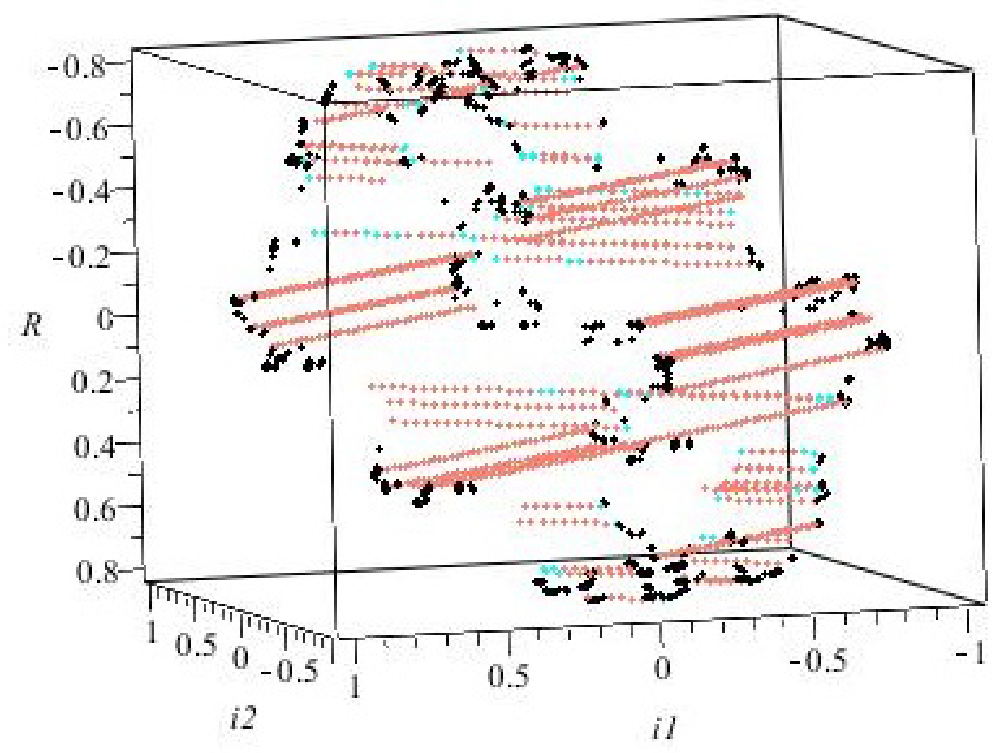}}
  \caption{Bicomplex Julia sets in 3D space}
  \label{ImJulia3d1}
\end{figure}

%--------------------------------------------------------------------------------------------------------
\subsection{Normal Families}
As in one complex variable, another way to define bicomplex Julia sets is to use the concept of normal families. Let us first recall the following definition of bicomplex normality introduced in \cite{charak}.

\begin{definition}  The family $\boldsymbol F$ of bicomplex holomorphic functions defined on a domain $D \subseteq \BC$ is said to be
\textbf{normal} in $D$ if every sequence in $\boldsymbol F$ contains a subsequence which on compact subsets of $D$ either:
1. converges uniformly to a limit function
or
2. converges uniformly to $\infty$.
The family $\boldsymbol F$ is said to be \textbf{normal at a point}
$z\in D$ if it is normal in some neighborhood of $z$ in $D$.
\label{compnormal}
\end{definition}
\begin{remark}
We say that the sequence $\{w_n\}$ of bicomplex numbers converges to
$\infty$ if and only if the norm $\{\|w_n\|\}$ converges to
$\infty$.
\end{remark}

With this definition we have this other possible definition of bicomplex Julia sets.

\begin{definition}
Let $P(w)$ be a bicomplex polynomial. We define the bicomplex Julia set for $P$ as
\begin{equation}
\mathcal{J}_{2}(P)=\{w\in\BC \mid \{P^{\circ n}(w)\}\mbox{ is not normal}\}.
\label{fund}
\end{equation}
\end{definition}

Moreover, from Theorem $11$ of \cite{charak}, we know that if the projection $\boldsymbol F_{ei}=\mathcal{P}_{i}(\boldsymbol F)$ is normal on $\mathcal{P}_{i}(D)$ for $i=1,2$ then the family of bicomplex holomorphic functions $\boldsymbol F$ is also normal on $D$. Hence, we obtain the following inclusion:
\begin{equation}
\mathcal{J}_{2}(P) \subset \mbox{\large \{ }z_1+z_2\bt\in\BC\mid \{[\mathcal{P}_1(P)]^{\circ n}(z_1-z_2\bo)\}\mbox{ or }
\{[\mathcal{P}_2(P)]^{\circ n}(z_1+z_2\bo)\}\mbox{ is not normal \large \}}.
\label{inc}
\end{equation}

From explicit counter examples (see \cite{charak}), we know also that (\ref{inc}) cannot be transformed into an equality.
In fact, in the particular case of the bicomplex holomorphic polynomial $P_c(w)=w^2+c$ where $\boldsymbol F:=\{P^{\circ n}_c,\mbox{ }\forall n\in\mathbb{N}\}$,
we have this following characterization of normality.

\begin{theorem}
Let $\boldsymbol F$ be a family of the iterates of the bicomplex polynomial $P_c(w)=w^2+c$. The family $\boldsymbol F$ is normal at $w_0\in\BC$ if and only if 1. $\boldsymbol F_{ei}=\mathcal{P}_{i}(\boldsymbol F)$ is normal at $\mathcal{P}_{i}(w_0)$ for $i=1,2$
or 2. $\mathcal{P}_{i}(w_0)$ is in the unbounded component of the Fatou set of $\boldsymbol F_{ei}$ for $i=1$ or $2$.
\label{normal3}
\end{theorem}
\begin{proof}
From Theorem $11$ of \cite{charak}, if $\boldsymbol F_{ei}$ is normal in a neighborhood $D_i$ of $\mathcal{P}_{i}(w_0)$ for  $i=1$ and $2$, then
$\boldsymbol F$ is normal in the following neighborhood $D:=D_1\times_e D_2$ of $w_0$. Moreover, for $i=1$ or $2$, supposed that every sequence in $\boldsymbol F_{ei}$ contains a subsequence which on compact subsets of a neighborhood of $\mathcal{P}_{i}(w_0)$ converges uniformly to $\infty$. Then, since one component can be chosen to converges uniformly to $\infty$, there exist a neighborhood of $w_0$ such that every sequence of $\boldsymbol F$ will also contains a subsequence which on compact subsets converges uniformly to $\infty$.

On the other side, we assume that the family $\boldsymbol F$ is normal at $w_0$. By the Theorem \ref{TheoInclBC}, it is always possible to consider a bicomplex neighborhood $D=D_1\times_e D_2$ of $w_0$ where $\boldsymbol F$ is normal. Now, without lost of generality, consider that $\mathcal{P}_{i}(w_0)\in\mK(\mathcal{P}_{i})$ for $i=1,2$ with $\mathcal{P}_{1}(w_0)\in\mathcal{J}(\mathcal{P}_{1})$. Therefore, there exist a sequence in $\{F_{1n}\}\in\boldsymbol F_{e1}$ and a compact set $K_1\in D_1$ such that no subsequence converges uniformly on $K_1$ to a $\mathbb{C}(\bo)$-function. Since $\boldsymbol F$ is normal then the sequence $\{F_n(w):=F_{1n}(\mathcal{P}_{1}(w))\eo+F_{2n}(\mathcal{P}_{2}(w))\et\}\in \boldsymbol F$, where $\{F_{2n}\}$ is the corresponding sequence in $\boldsymbol F_{e2}$, has a subsequence $\{F_{n_k}\}$ which converges uniformly on compact subset to either a $\BC$-function or to $\infty$. However, no subsequence of $\{F_n(w)\}$ can converge uniformly on $K:=K_1\times_e K_2$, where $K_2$ is any compact set in $D_2$. Therefore, $\{F_n(w)\}$ should have a subsequence $\{F_{n_k}\}$ which converges uniformly on compact subset to $\infty$. This is a contradiction since $\{w_0\}=\mathcal{P}_{1}(w_0)\times_e \mathcal{P}_{2}(w_0)$ is a compact subset of $D$ where, by definition of the filled-in Julia sets, the sequence $\{F_n(w_0)\}$ is bounded in each component.
\end{proof}$\Box$\\

Hence, we obtain automatically this following characterization of bicomplex Julia sets.

\begin{corollary}
Let $P_c(w)=w^2+c$, then
\begin{equation}
\mathcal{J}_{2}(P_c)=\pa \mK_{2,c}=\mJ_{2,c}.
\end{equation}
\label{filled}
\end{corollary}

We conclude this section with this following remark from \cite{charak} concerning the case of non-degenerate polynomials of degree greater than one.

\begin{remark}
The Corollary \ref{filled} can be extended for any bicomplex polynomials of the form $P(w)=a_dw^d+a_{d-1}w^{d-1}+...+a_0$ where $a_d\notin
\mathcal{NC}$ and $d\geq 2$.
\end{remark}

%--------------------------------------------------------------------------------------------------------
\section{The Inverse Iteration Method}

To extend the inverse iteration method for bicomplex Julia sets, we need to find a result similar to the second part of Theorem \ref{TheoInvC} for $\mJ_{2,c}$. However, the structure of this set as expressed in Corollary \ref{corJulia} suggests it is impossible to do so. Indeed, the set of inverse iterates of $P_c$ cannot be dense in $\mJ_{2,c}$ due to the presence of $\mK_{c_1 - c_2 \bo}$ and $\mK_{c_1 + c_2 \bo}$ (where the complex inverse iterates are not usually dense). Nevertheless, it appears relevant to adapt the inverse iteration method for $\mJ_{c_1 - c_2 \bo} \times_e \mJ_{c_1 + c_2 \bo}$ since
this specific case coincide with this specific condition :
$$
\mbox{\large \{ }z_1+z_2\bt\in\mJ_{2,c}\mid \{[\mathcal{P}_1(w^2+c)]^{\circ n}(z_1-z_2\bo)\}\mbox{ and }
\{[\mathcal{P}_2(w^2+c)]^{\circ n}(z_1+z_2\bo)\}\mbox{ are not normal \large \}}.
$$
We examine the bicomplex fixed points of $P_c$ and relate them to $\mJ_{c_1 - c_2 \bo}$ and $\mJ_{c_1 + c_2 \bo}$ to achieve our goal.

\begin{lemma}
  Let $c = c_1 + c_2 \bt \in \BC$. Then $p = p_1 \eo + p_2 \et \in \BC$ is a fixed point of $P_c$ if and only if $p_1$ is a fixed point of $P_{c_1 - c_2 \bo}$ and $p_2$ is a fixed point of $P_{c_1 + c_2 \bo}$.
\end{lemma}

\begin{proof}
Let $c = c_1 + c_2 \bt \in \BC$ and $p = p_1 \eo + p_2 \et \in \BC$. Then $p$ is a fixed point of $P_c$ if and only if
\begin{eqnarray*}
  p^2 + c = p & \Leftrightarrow & (p_1 \eo + p_2 \et)^2 + (c_1 - c_2 \bo)\eo + (c_1 + c_2 \bo)\et = p_1 \eo + p_2 \et\\
   & \Leftrightarrow &
   \begin{cases}
       p_1 ^2 + (c_1 - c_2 \bo) = p_1 \\
       p_2 ^2 + (c_1 + c_2 \bo) = p_2
    \end{cases}\\
    & \Leftrightarrow &
    \begin{cases}
      P_{c_1 - c_2 \bo} (p_1) = p_1\\
      P_{c_1 + c_2 \bo} (p_2) = p_2.
    \end{cases}
\end{eqnarray*}
Hence, $p$ is a fixed point of $P_c$ if and only if $p_1$ is a fixed point of $P_{c_1 - c_2 \bo}$ and $p_2$ is a fixed point of $P_{c_1 + c_2 \bo}$.
\end{proof}$\Box$\\

\begin{lemma}
  Let $c = c_1 + c_2 \bt \in \BC$. $P_c$ has at least one fixed point in $\mJ_{c_1 - c_2 \bo} \times_e \mJ_{c_1 + c_2 \bo}$.
\end{lemma}

\begin{proof}
Let $c = c_1 + c_2 \bt \in \BC$ and consider the fixed points associated to $P_{c_1 - c_2 \bo}$ and $P_{c_1 + c_2 \bo}$. If $c_1 - c_2 \bo \neq \frac{1}{4}$, then $P_{c_1 - c_2 \bo}$ has at least one repelling fixed point in $\mJ_{c_1 - c_2 \bo}$. If $c_1 - c_2 \bo = \frac{1}{4}$, then there is a unique indifferent fixed point and it also belongs to $\mJ_{c_1 - c_2 \bo}$. Since the same reasoning applies for $c_1 + c_2 \bo$, then $P_{c_1 - c_2 \bo}$ has at least one fixed point $p_1 \in \mJ_{c_1 - c_2 \bo}$ and $P_{c_1 + c_2 \bo}$ has at least one fixed point $p_2 \in \mJ_{c_1 + c_2 \bo}$. Thus, there exists a fixed point $p = p_1 \eo + p_2 \et$ of $P_c$ such that $p \in \mJ_{c_1 - c_2 \bo} \times_e \mJ_{c_1 + c_2 \bo}$.
\end{proof}$\Box$\\

\noindent Now, let us state and demonstrate the main result of this article.

\begin{theorem}
  Let $c = c_1 + c_2 \bt \in \BC$ and $w_1$ a fixed point of $P_c$ such that $w_1 \in \mJ_{c_1 - c_2 \bo} \times_e \mJ_{c_1 + c_2 \bo}$. The set $\oa \bigcup\limits_{k=1}^{\infty} P_c ^{-k}(w_1)\fa$ is dense in $\mJ_{c_1 - c_2 \bo} \times_e \mJ_{c_1 + c_2 \bo}$.
\end{theorem}

\begin{proof}
Let $c = c_1 + c_2 \bt \in \BC$ and $w_1 = s_1 + s_2 \bt \in \BC$ such that $P_c (w_1) = w_1$ and $w_1 \in \mJ_{c_1 - c_2 \bo} \times_e \mJ_{c_1 + c_2 \bo}$. We need to show that for all $w \in \mJ_{c_1 - c_2 \bo} \times_e \mJ_{c_1 + c_2 \bo}$, any open set containing $w$ also contains an element of $\oa \bigcup\limits_{k=1}^{\infty} P_c ^{-k}(w_1)\fa$. Let $w = z_1 + z_2 \bt \in \mJ_{c_1 - c_2 \bo} \times_e \mJ_{c_1 + c_2 \bo}$ and $V \subseteq \BC$ an open set with $w \in V$. By hypothesis, we have $s_1 - s_2 \bo \in \mJ_{c_1 - c_2 \bo}$ and $s_1 + s_2 \bo \in \mJ_{c_1 + c_2 \bo}$. From Theorem \ref{TheoInvC}, this means that
\begin{description}
  \item \emph{(a)} $\oa \bigcup\limits_{k=k_1}^{\infty} P_{c_1 - c_2 \bo} ^{-k}(s_1 - s_2 \bo)\fa$ is dense in $\mJ_{c_1 - c_2 \bo}$ for all $k_1 \geq 1$
  \item \emph{(b)} $\oa \bigcup\limits_{k=k_1}^{\infty} P_{c_1 + c_2 \bo} ^{-k}(s_1 + s_2 \bo)\fa$ is dense in $\mJ_{c_1 + c_2 \bo}$ for all $k_1 \geq 1$.
\end{description}
Since $V$ is open, there exists $R>0$ such that $B^2 (w,R) \subseteq V$. From Theorem \ref{TheoInclBC}, $D(w;R,R) = B^1 (z_1 - z_2 \bo, R) \times_e B^1 (z_1 + z_2 \bo, R) \varsubsetneq B^2 (w,R)$. Considering that $B^1 (z_1 - z_2 \bo,R)$ is open and centered in $\mJ_{c_1 - c_2 \bo}$, then it must contain an element of $\oa \bigcup\limits_{k=1}^{\infty} P_{c_1 - c_2 \bo} ^{-k}(s_1 - s_2 \bo)\fa$ from \emph{(a)} with $k_1 = 1$. Hence, there exists $v_1 \in B^1 (z_1 - z_2 \bo,R)$ such that $v_1 \in P_{c_1 - c_2 \bo} ^{-k'} (s_1 - s_2 \bo)$ for some $k' \geq 1$. Similarly, $B^1 (z_1 + z_2 \bo, R)$ contains an element of $\oa \bigcup\limits_{k=k'}^{\infty} P_{c_1 + c_2 \bo} ^{-k}(s_1 + s_2 \bo)\fa$ from \emph{(b)} with $k_1 = k'$. So, there exists $v_2 \in B^1 (z_1 + z_2 \bo, R)$ with $v_2 \in P_{c_1 + c_2 \bo} ^{-k''} (s_1 + s_2 \bo)$ for some $k'' \geq k' \geq 1$. There are two possible scenarios :
\begin{description}
  \item \emph{(i)} Suppose $k'' = k'$. Then there exists $v = v_1 \eo + v_2 \et$ such that $v \in V$ and $v \in P_{c_1 - c_2 \bo} ^{-k'} (s_1 - s_2 \bo) \times_e P_{c_1 + c_2 \bo} ^{-k'} (s_1 + s_2 \bo) = P_c ^{-k'}(w_1)$ for some $k' \geq 1$.
  \item \emph{(ii)} Suppose $k'' > k'$. Since $w_1$ is a fixed point of $P_c$, then $s_1 - s_2 \bo$ is a fixed point of $P_{c_1 - c_2 \bo}$. Using induction, we easily show that $s_1 - s_2 \bo = P_{c_1 - c_2 \bo} ^n (s_1 - s_2 \bo)$ for all $n \in \{1,2, \dots\}$. Since, $v_1 \in P_{c_1 - c_2 \bo} ^{-k'} (s_1 - s_2 \bo)$, we may write $P_{c_1 - c_2 \bo} ^{k'} (v_1) = s_1 - s_2 \bo$. Let $m = k'' - k' \geq 1$. Then
      $$P_{c_1 - c_2 \bo} ^{k''}(v_1) = P_{c_1 - c_2 \bo} ^{m + k'} (v_1) =P_{c_1 - c_2 \bo} ^{m} \(P_{c_1 - c_2 \bo} ^ {k'} (v_1)\) = P_{c_1 - c_2 \bo}^{m} (s_1 - s_2 \bo) = s_1 - s_2 \bo$$
      and $v_1 \in P_{c_1 - c_2 \bo} ^{-k''}(s_1 - s_2 \bo)$. Therefore there exists $v = v_1 \eo + v_2 \et \in V$ such that $v \in P_c ^{-k''} (w_1)$ for some $k'' > 1$.
\end{description}
In both cases, there exists $v \in V$ such that $v \in \oa \bigcup\limits_{k=1}^{\infty} P_c ^{-k}(w_1)\fa$. Thus if $w_1$ is a fixed point of $P_c$ in $\mJ_{c_1 - c_2 \bo} \times_e \mJ_{c_1 + c_2 \bo}$, then $\oa \bigcup\limits_{k=1}^{\infty} P_c ^{-k}(w_1)\fa$ is dense in $\mJ_{c_1 - c_2 \bo} \times_e \mJ_{c_1 + c_2 \bo}$.
\end{proof}$\Box$\\

\begin{figure}[!h]
  \centering
  \includegraphics[width=8cm]{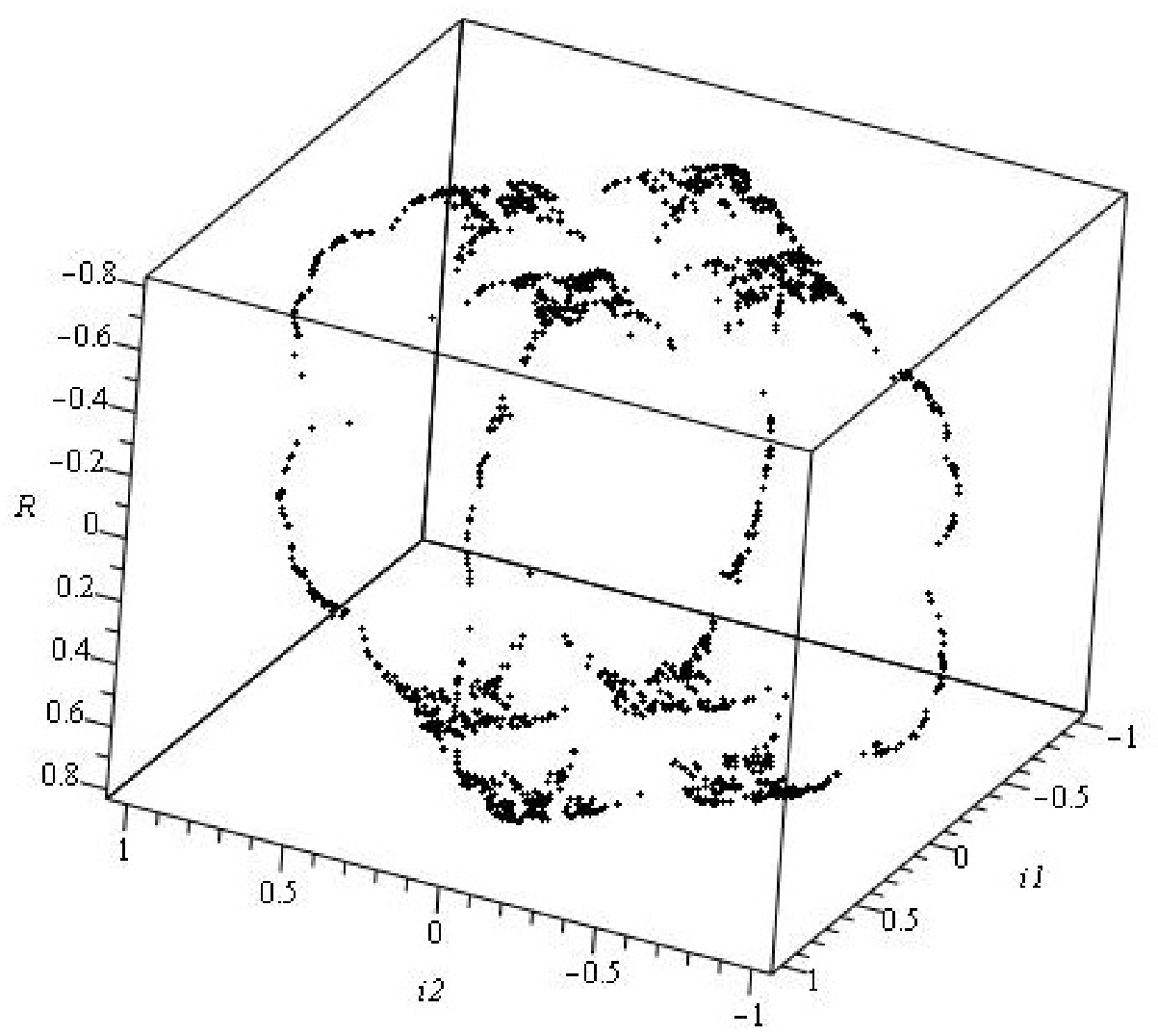}
  \caption{$\mJ_{c_1 - c_2 \bo} \times_e \mJ_{c_1 + c_2 \bo}$ for $c=0,25$ in 3D space}
  \label{ImJulia3dInverse1}
\end{figure}

Consequently, part of a bicomplex Julia set may be created using an adapted version of the inverse iteration method. As an example, for $c = 0,25$ the set $\mJ_{c_1 - c_2 \bo} \times_e \mJ_{c_1 + c_2 \bo}$ is generated and given in Figure \ref{ImJulia3dInverse1}. Since $c_1 - c_2 \bo = c_1 + c_2 \bo = 0,25$ and $\frac{1}{2}$ is the only fixed point of $P_{c_1-c_2 \bo} = P_{c_1+c_2 \bo}$, $w_1 = \frac{1}{2} \eo + \frac{1}{2} \et$ is chosen to start the algorithm. At each iteration, the inverse is computed choosing randomly one of the branches of the bicomplex square root function $\sqrt{w-c}$. A 3D cut is applied as before, by keeping only the points for which the absolute value of the $\bj$-component is less than a certain value $\epsilon > 0$.\\

In the complex plane, a filled-in Julia set that is a dendrite has the property of being equal to its boundary : $\mK_c = \mJ_c$. A \textbf{bicomplex dendrite} may be defined as a bicomplex cartesian set for which both components are dendrites in $\mC(\bo)$. Hence, $\mK_{2,c}$ is a dendrite when $\mK_{c_1 - c_2 \bo}$ and $\mK_{c_1 + c_2 \bo}$ are complex dendrites. In this case, $\mK_{2,c} = \mJ_{c_1 - c_2 \bo} \times_e \mJ_{c_1 + c_2 \bo} = \mJ_{2,c}$ and the last theorem directly applies to observe the complete bicomplex Julia set and not only part of it.

\begin{figure}[!h]
  \centering
  \includegraphics[width=8cm]{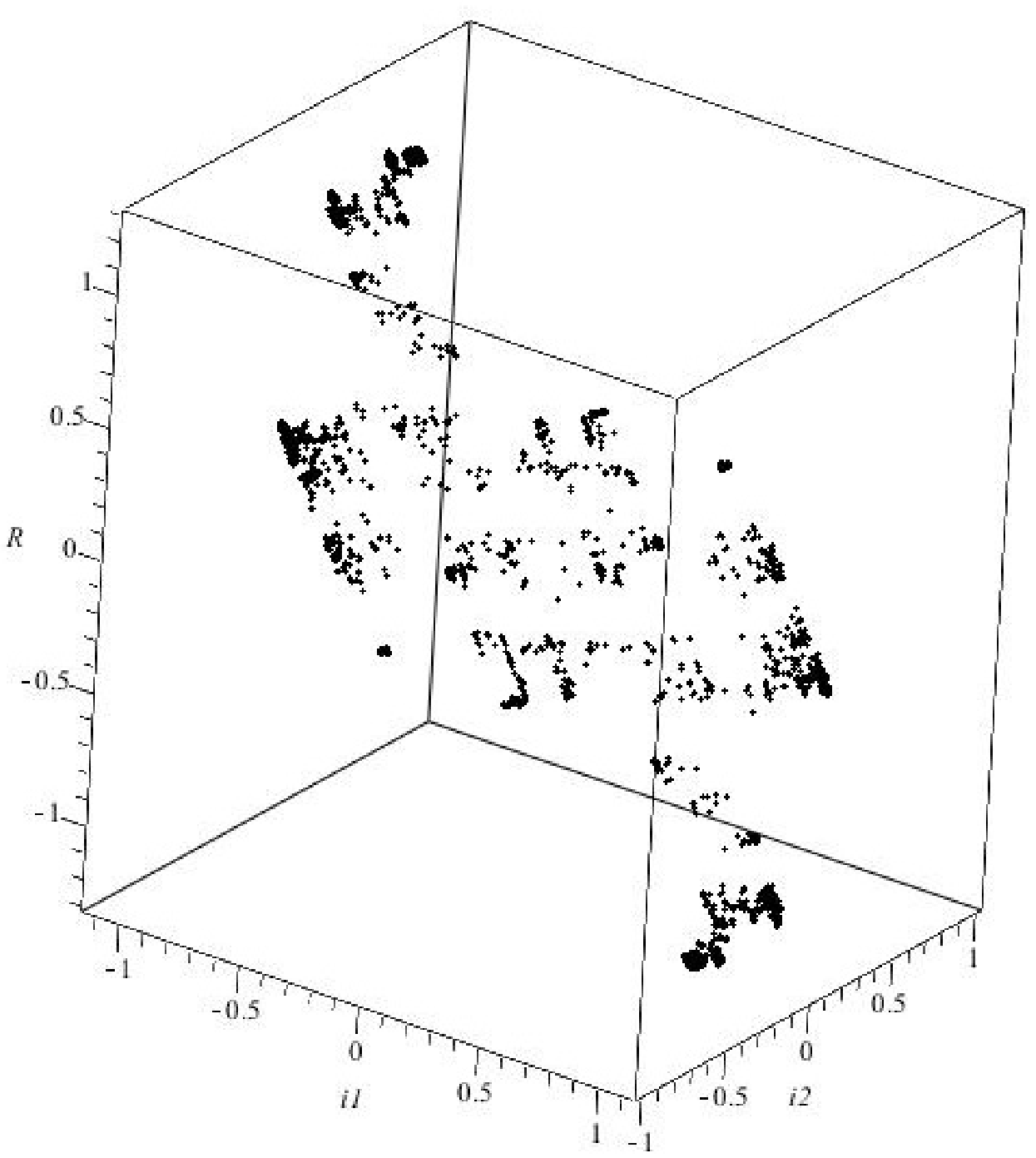}
  \caption{$\mJ_{2,c}$ for $c= \bo$ in 3D space}
  \label{ImJulia3dInverse2}
\end{figure}

\begin{corollary}
  Let $c = c_1 + c_2 \bt \in \BC$ and $\mK_{2,c}$ a bicomplex dendrite. Let $w_1$ a fixed point of $P_c$ such that $w_1 \in \mJ_{2,c}$. The set $\oa \bigcup\limits_{k=1}^{\infty} P_c ^{-k}(w_1)\fa$ is dense in $\mJ_{2,c}$.
\end{corollary}

Consider $c = \bo$, that is $c = c_1 + c_2 \bt \in \BC$ with $c_1 = \bo$ and $c_2 = 0$. Then $c_1 - c_2 \bo = c_1 + c_2 \bo = \bo$ and $\mK_{2,c} = \mJ_{2,c} = \mJ_{\bo} \times \mJ_{\bo}$ is a dendrite. Again, the inverse iteration method is adaptable to generate this set and display it in the 3D space. To do so, let $z_1 \in \mJ_{\bo}$ such that $z_1 = P_{\bo} (z_1)$. Then $w_0 = z_1 \eo + z_1 \et$ is a fixed point of $P_c$ in $\mJ_{2,c}$ and so, a good starter for the algorithm. The same reasoning as before is used to produce Figure \ref{ImJulia3dInverse2}.

%-------------------------------------------------------------------------------------------------------

\section*{Acknowledgments}
DR is grateful to the Natural Sciences and
Engineering Research Council of Canada for financial
support.  CM would like to thank the Qu\'{e}bec
FRQNT Fund for the award of a postgraduate
scholarship.

%-----------------------------------------------------------------------------------------------

\end{document}